\documentclass[12pt]{amsart}
\usepackage{amssymb}

\newtheorem{theorem}{Theorem}[section]
\newtheorem{proposition}[theorem]{Proposition}
\newtheorem{lemma}[theorem]{Lemma}

\newtheorem{remark}[theorem]{Remark}

\numberwithin{equation}{section}

\begin{document}

\baselineskip=15pt

\title[Nef cone of flag bundles over a curve]{Nef cone
of flag bundles over a curve}

\author[I. Biswas]{Indranil Biswas}

\address{School of Mathematics, Tata Institute of Fundamental
Research, Homi Bhabha Road, Bombay 400005, India}

\email{indranil@math.tifr.res.in}

\author[A.J. Parameswaran]{A. J. Parameswaran}

\address{School of Mathematics, Tata Institute of Fundamental
Research, Homi Bhabha Road, Bombay 400005, India}

\email{param@math.tifr.res.in}



\date{}

\begin{abstract}
Let $X$ be a smooth projective curve defined over an algebraically
closed field $k$, and let $E$ be a vector bundle on $X$. Let
${\mathcal O}_{\text{Gr}_r(E)}(1)$ be the tautological line bundle over 
the Grassmann bundle $\text{Gr}_r(E)$ parametrizing all the $r$ dimensional 
quotients of
the fibers of $E$. We give necessary and sufficient conditions for
${\mathcal O}_{\text{Gr}_r(E)}(1)$ to be ample and nef respectively.
As an application, we compute the nef cone of $\text{Gr}_r(E)$.
This yields a description of the nef cone of any flag bundle
over $X$ associated to $E$.
\end{abstract}

\maketitle

\section{Introduction}

Let $E$ be a semistable vector bundle over a smooth projective curve defined over
an algebraically closed field of characteristic zero. Miyaoka computed the nef
cone of ${\mathbb P}(E)$ \cite[p. 456, Theorem 3.1]{Mi}. Our aim here is to
compute the nef cone of the flag bundles associated to vector bundles
over curves.

Let $X$ be an irreducible smooth projective curve defined
over an algebraically closed field $k$ (the characteristic is not necessarily
zero).
If the characteristic of $k$ is positive, the
absolute Frobenius morphism of $X$ will be denoted by $F_X$.
A vector bundle $E$ on $X$ is called strongly semistable if all
the pullbacks of $E$ by the iterations of $F_X$ are semistable.

Let $E$ be a vector bundle on $X$. Let
\begin{equation}\label{ei1}
E_1\, \subset\, E_2 \, \subset\, \cdots \, \subset\,
E_{m-1} \, \subset\, E_m \,=\, E
\end{equation}
be the Harder--Narasimhan filtration of $E$. If the
characteristic of $k$ is zero, and
$$
f\, :\, Y\, \longrightarrow\, X
$$
is a nonconstant morphism, where $Y$ is an irreducible
smooth projective curve, then the pulled back filtration
$$
f^*E_1\, \subset\, f^*E_2 \, \subset\, \cdots \, \subset\,
f^*E_{m-1} \, \subset\, f^*E_m \,=\, f^*E
$$
coincides with the Harder--Narasimhan filtration of $f^*E$.
If the characteristic of $k$ is positive, then this is not
true in general. However, there is an integer $n_E$, that depends
on $E$, such that the Harder--Narasimhan filtration of
$(F^n_X)^*E$ has this property if $n\, \geq\, n_E$, meaning the Harder--Narasimhan 
filtration of $f^*(F^{n}_X)^*E$ is the pullback, by $f$, of
the Harder--Narasimhan filtration of $(F^{n}_X)^*E$, where $f$
is any nonconstant morphism to $X$ from an
irreducible smooth projective curve.

Fix an integer $r\,\in\, [1\, ,\text{rank}(E)-1]$. Let
${\rm Gr}_r(E)$ be the Grassmann bundle on $X$ parametrizing all the
$r$ dimensional quotients of the fibers of $E$. The tautological
line bundle on ${\rm Gr}_r(E)$ will be denoted by
${\mathcal O}_{\text{Gr}_r(E)}(1)$.

If the characteristic of $k$ is positive,
consider the Harder--Narasimhan filtration of $(F^{n_E}_X)^*E$
$$
0\, =\, V_0\, \subset\, V_1\, \subset\, V_2 \, \subset\, \cdots \, 
\subset\, V_{d-1} \, \subset\, V_d \,=\, (F^{n_E}_X)^*E\, ,
$$
where $n_E$ is as above; if the characteristic of $k$ is
zero, then simply take the Harder--Narasimhan filtration of $E$.
So $V_i$ is $E_i$ in \eqref{ei1} if the characteristic of $k$ is
zero. Using only the numerical data associated to this filtration, we
can compute a rational number
$\theta_{E,r}$ (see \eqref{e4}). The following theorem shows that
$\theta_{E,r}$ controls the positivity of the tautological line bundle
${\mathcal O}_{{\rm Gr}_r(E)}(1)$ on ${\rm Gr}_r(E)$.

\begin{theorem}\label{thm0}
If $\theta_{E,r}\, >\, 0$, then the tautological line bundle
${\mathcal O}_{{\rm Gr}_r(E)}(1)$ is ample.

If $\theta_{E,r}\, =\, 0$, then 
${\mathcal O}_{{\rm Gr}_r(E)}(1)$ is nef but not ample.

If $\theta_{E,r}\, <\, 0$, then ${\mathcal O}_{{\rm Gr}_r(E)}(1)$
is not nef.
\end{theorem}

(See Theorem \ref{thm1} for a proof of the above theorem.)

As an application of Theorem \ref{thm0}, we compute the
nef cone of ${\rm Gr}_r(E)$ (this is done in Section \ref{s-n-c}).

In order to know the nef cone of a flag bundle over $X$ 
associated to $E$, it is enough to know the nef cones of the
corresponding Grassmann bundles associated to $E$. Therefore, 
using our description of the nef cone of the Grassmann bundles
we obtain a description of the nef cone of any flag bundle over 
$X$ associated to $E$; see Theorem \ref{th-fl}.

Let $K^{-1}_\varphi\, :=\, K^{-1}_{{\rm Gr}_r(E)}\bigotimes 
\varphi^*K_X$
be the relative anti-canonical line bundle for the natural projection
$\varphi\, :\, {\rm Gr}_r(E)\, \longrightarrow\, X$. It is known 
that $K^{-1}_\varphi$ is never ample. If the characteristic of 
$k$ is zero, then $K^{-1}_\varphi$ is nef if and only if $E$ is
semistable \cite{BB}; if the characteristic of $k$ is positive,
then $K^{-1}_\varphi$ is nef if and only if $E$ is strongly 
semistable \cite{BH}. These criteria for semistability and strong semistability
follow from the description of the nef cone of ${\rm Gr}_r(E)$
given in Proposition \ref{prop-e1} and Proposition \ref{prop-e2}.

\section*{Acknowledgements}
We are very grateful to the referee for going though the paper very
carefully and providing detailed comments to improve the exposition.

\section{Preliminaries}

Let $k$ be an algebraically closed field. Let $X$ be an 
irreducible smooth projective curve defined over $k$. If the
characteristic of $k$ is positive, then we have the absolute
Frobenius morphism
$$
F_X\, :\, X \longrightarrow\, X\, .
$$
For convenience, if the characteristic of $k$ is zero, by
$F_X$ we will denote the identity morphism of $X$. For any
integer $m\, \geq\, 1$, let
$$
F^m_X \, :=\, \overbrace{F_X\circ\cdots\circ
F_X}^{m\mbox{-}\rm{times}}\, :\, X\, \longrightarrow\, X
$$
be the $m$--fold iteration of $F_X$. For
notational convenience, by $F^0_X$ we will denote the
identity morphism of $X$.

For a vector bundle $E$ over $X$ of positive rank, define the
number
$$
\mu(E)\, :=\, \frac{\text{degree}(E)}{\text{rank}(E)}
\,\in\, \mathbb Q\, .
$$

A vector bundle $E$ over $X$ is called \textit{semistable}
if for every nonzero subbundle $V\, \subset\, E$, the
inequality
$$
\mu(V)\, \leq\, \mu(E)
$$
holds. The vector bundle $E$ is called \textit{strongly
semistable} if the pullback $(F^m_X)^*E$ is semistable for
all $m\, \geq\, 0$.

For every vector bundle $E$ on $X$, there is a unique filtration 
of subbundles
$$
0\, =\, E_0\, \subset\, E_1\, \subset\, \cdots \, \subset\,
E_{d_E-1} \, \subset\, E_{d_E} \,=\, E
$$
such that $E_i/E_{i-1}$ is semistable for each
$i\, \in\, [1\, ,d_E]$, and $\mu(E_i/E_{i-1})\, > \, 
\mu(E_{i+1}/E_{i})$ for all $i\, \in\, [1\, ,d_E-1]$.
It is known as the \textit{Harder--Narasimhan filtration} of $E$.
If $E$ is semistable, then $d_E\,=\, 1$.

Given any $E$, there is a nonnegative integer $\delta$ 
satisfying the condition that for all $i\, \geq\, 1$,
\begin{equation}\label{e1}
0\, =\, (F^i_X)^* V_0\, \subset\, (F^i_X)^* V_1\, \subset\, 
\cdots \, \subset\,
(F^i_X)^* V_{d-1} \, \subset\, (F^i_X)^* V_{d} \,=\, 
(F^{i+\delta}_X)^*E
\end{equation}
is the Harder--Narasimhan filtration of $(F^{i+\delta}_X)^*E$,
where
\begin{equation}\label{e-1}
0\, =\, V_0\, \subset\, V_1\, \subset\, \cdots \, \subset\,
V_{d-1} \, \subset\, V_{d} \,=\, (F^\delta_X)^*E
\end{equation}
is the Harder--Narasimhan filtration of $(F^\delta_X)^*E$
\cite[p. 259, Theorem 2.7]{La} (this is vacuously true if
the characteristic of $k$ is zero). It should be 
emphasized that $\delta$ in \eqref{e1} depends on $E$.

Note that the quotient $V_i/V_{i-1}$ in the filtration in
\eqref{e-1} is strongly semistable
for all $i\, \in\, [1\, ,d]$. If $\delta$ satisfies the above
condition, then clearly $\delta+j$ also satisfies the above
condition for all $j\, \geq\, 0$.

For a vector bundle $E$ on $X$, let ${\mathbb P}(E)$ denote the
projective bundle over $X$ parametrizing all the hyperplanes
in the fibers of $E$. The vector bundle $E$ is called 
\textit{ample} if the tautological line bundle
${\mathcal O}_{{\mathbb P}(E)}(1)$ on ${\mathbb P}(E)$ is ample
(see \cite{Ha} for properties of ample bundles).

A line bundle $L$ over an irreducible projective variety $Z$
defined over $k$ is
called \textit{numerically effective} (``nef'' for short)
if for all pairs of the form $(C\, ,f)$, where $C$ is a smooth
projective curve, and $f$ is a morphism from $C$ to $Z$, the inequality
$$
\text{degree}(f^*L) \, \geq\, 0
$$
holds. A vector bundle $E$ is called \textit{nef} if the
tautological line bundle ${\mathcal O}_{{\mathbb P}(E)}(1)$
over ${\mathbb P}(E)$ is nef.

The following lemma is well-known.

\begin{lemma}\label{lem1}
Let $0\, \longrightarrow\, W \, \longrightarrow\, E
\, \longrightarrow\, Q\, \longrightarrow\, 0$ be a short exact
sequence of vector bundles. If both $W$ and $Q$ are ample (respectively,
nef), then $E$ is ample (respectively, nef).
\end{lemma}

See \cite[p. 71, Corollary 3.4]{Ha} for the case of ample bundles, and
\cite[p. 308, Proposition 1.15(ii)]{DPS} for the case of nef vector bundles.

\section{(Semi)positivity criterion}\label{sec3}

Let $E$ be a vector bundle over $X$ of rank at least two.
Fix an integer $r\, \in\, [1\, ,\text{rank}(E)-1]$. Let
\begin{equation}\label{vp}
\varphi\, :\, \text{Gr}_r(E) \, \longrightarrow\, X
\end{equation}
be the Grassmann bundle over $X$ parametrizing all the quotients,
of dimension $r$, of the fibers of $E$. Let
\begin{equation}\label{f1}
{\mathcal O}_{\text{Gr}_r(E)}(1)\, \longrightarrow\, 
\text{Gr}_r(E)
\end{equation}
be the tautological line bundle; the fiber of ${\mathcal 
O}_{\text{Gr}_r(E)}(1)$ over any quotient $Q$ of $E_x$ is
$\bigwedge^r Q$. So the line bundle ${\mathcal O}_{\text{Gr}_r(E)}(1)$ is
relatively ample.

Take any $\delta$ satisfying the condition in \eqref{e1}. Let
\begin{equation}\label{e2}
0\, =\, V_0\, \subset\, V_1\, \subset\, \cdots \, \subset\,
V_{d-1} \, \subset\, V_{d} \,=\, (F^\delta_X)^*E
\end{equation}
be the Harder--Narasimhan filtration of $(F^\delta_X)^*E$.
We recall that $V_i/V_{i-1}$ is strongly semistable
for all $i\, \in\, [1\, ,d]$. Let
$$
t\, \in\, [1\, , d]
$$
be the unique largest integer such that
\begin{equation}\label{e3}
\sum_{i=t}^d \text{rank}(V_i/V_{i-1})\, \geq \, r\, ;
\end{equation}
so either $t\,=\, d$, or $t$ is the smallest
integer with
$$
\sum_{i=t+1}^d \text{rank}(V_i/V_{i-1})\,=\,
\text{rank}(((F^\delta_X)^*E)/V_t)\, < \, r\, .
$$

Define
\begin{equation}\label{e4}
\theta_{E,r}\, :=\, (r- \text{rank}(((F^\delta_X)^*E)/V_t))
\cdot \mu(V_t/V_{t-1}) + \text{degree}(((F^\delta_X)^*E)/V_t)\, ,
\end{equation}
where $t$ is defined above using \eqref{e3}. If $E$ is strongly semistable,
then we may take $\delta\,=\, 1$; in that case, $\theta_{E,r}\,=\,
r\cdot \mu(E)$. Note that the condition
that $\theta_{E,r}$ is nonzero, or the condition that $\theta_{E,r}$
is positive, does not depend on the choice of
the integer $\delta$ in \eqref{e2}.

\begin{lemma}\label{thm1-l1}
Assume that $\theta_{E,r}\, >\, 0$. Then the line bundle
${\mathcal O}_{{\rm Gr}_r(E)}(1)\, \longrightarrow\,
{\rm Gr}_r(E)$ in \eqref{f1} is ample.
\end{lemma}

\begin{proof}
Consider the Pl\"ucker embedding
\begin{equation}\label{rho}
\rho\, :\, \text{Gr}_r(E)\, \longrightarrow\, {\mathbb P}
(\bigwedge\nolimits^r E)\, .
\end{equation}
We have
\begin{equation}\label{rho1}
\rho^*{\mathcal O}_{{\mathbb P} (\bigwedge^r E)}(1)\,=\,
{\mathcal O}_{\text{Gr}_r(E)}(1)\, .
\end{equation}
Therefore, to prove that
${\mathcal O}_{\text{Gr}_r(E)}(1)$ is ample, it suffices to
show that the vector bundle $\bigwedge^r E$ is ample. Since
$F^\delta_X$ is a finite flat surjective morphism, it follows 
that $\bigwedge^r E$ is ample if and only if
$(F^\delta_X)^*\bigwedge^r E$ is ample \cite[p. 73,
Proposition 4.3]{Ha}.

Using the filtration in \eqref{e2} it follows that the
vector bundle $(F^\delta_X)^*\bigwedge^r E$ admits a
filtration of subbundles such that each successive quotient is of the
form
\begin{equation}\label{ua}
V_{\underline a}\, :=\, \bigotimes_{i=1}^d 
\bigwedge\nolimits^{a_i}(V_i/V_{i-1})
\end{equation}
with $\sum_{i=1}^d a_i\,=\, r$; we use the standard convention 
that $\bigwedge\nolimits^{0} F$ is the trivial line bundle for
every vector bundle $F$. Since each $V_i/V_{i-1}$ is strongly
semistable, the above vector bundle $V_{\underline a}$ is
also strongly semistable (see
\cite[p. 285, Theorem 3.18]{RR} for ${\rm Char}(k)\,=\, 0$, and
\cite[p. 288, Theorem 3.23]{RR} for ${\rm Char}(k)\,>\, 0$). From
the assumption that $\theta_{E,r}\, >\, 0$ it follows immediately that
\begin{equation}\label{dva}
\text{degree}(V_{\underline a}) \, >\, 0\, .
\end{equation}

Since $V_{\underline a}$ is strongly semistable of positive
degree, it can be shown that $V_{\underline a}$ is ample \cite{BP}.
We include the details for completeness.

To prove that $V_{\underline a}$ is ample,
we need to show that for any coherent sheaf $\mathcal E$ on $X$, there
is a positive integer $b_{\mathcal E}$ such that
\begin{equation}\label{ch-am}
H^1(X,\, \text{Sym}^j(V_{\underline a})\otimes {\mathcal E})\,=\, 0
\end{equation}
for all $j\, \geq\, b_{\mathcal E}$ \cite[p. 70, Proposition 3.3]{Ha}.
Since $H^1(X,\, \text{Sym}^j(V_{\underline a})\bigotimes {\mathcal E})\,=\, 0$
if ${\mathcal E}$ is a torsion sheaf, and any vector bundle on $X$ admits
a filtration of subbundles
such that each successive quotient is a line bundle, it is enough
to prove \eqref{ch-am} for all line bundles ${\mathcal E}$. Take a
line bundle ${\mathcal E}$. Since $V_{\underline a}$ is strongly semistable,
it follows that $\text{Sym}^j(V_{\underline a})$ is semistable for all $j
\,\geq\, 1$ (see 
\cite[p. 285, Theorem 3.18]{RR} for ${\rm Char}(k)\,=\, 0$, and
\cite[p. 288, Theorem 3.23]{RR} for ${\rm Char}(k)\,>\, 0$). Therefore,
the vector bundle $\text{Sym}^j(V_{\underline a})^*\bigotimes {\mathcal E}^*
\bigotimes K_X$ is semistable. Now, from \eqref{dva} we conclude that
$$
\mu(\text{Sym}^j(V_{\underline a})^*\otimes {\mathcal E}^*
\otimes K_X)\,=\, -j\cdot \mu(V_{\underline a}) - \text{degree}({\mathcal E})
+2(\text{genus}(X)-1) \, <\, 0
$$
for all $j$ sufficiently large positive. Consequently,
$$
H^0(X,\, \text{Sym}^j(V_{\underline a})^*\otimes {\mathcal E}^*
\otimes K_X)\,=\, 0
$$
for all $j$ sufficiently large positive. Therefore, from Serre duality,
$$
H^1(X,\, \text{Sym}^j(V_{\underline a})\otimes {\mathcal E})\,=\, 0
$$
for all $j$ sufficiently large positive. Hence $V_{\underline a}$ is ample.

We note that if
the characteristic of $k$ is zero, then the nef cone of the projective bundle
${\mathbb P}(V_{\underline a})$ is explicitly described in \cite[p. 456,
Theorem 3.1(4)]{Mi}. It is straightforward to check that the tautological line
bundle ${\mathcal O}_{{\mathbb P}(V_{\underline a})}(1)$ lies in the interior
of the nef cone of ${\mathbb P}(V_{\underline a})$. This also proves that
$V_{\underline a}$ is ample under the assumption that the characteristic of $k$ is zero.

Since $V_{\underline a}$ is ample, and
$(F^\delta_X)^*\bigwedge^r E$ admits a filtration of subbundles such that
each successive quotient is of the form $V_{\underline a}$,
using Lemma \ref{lem1} we conclude that the vector bundle
$(F^\delta_X)^*\bigwedge^r E$ is ample. We noted earlier that
${\mathcal O}_{\text{Gr}_r(E)}(1)$ is ample if
$(F^\delta_X)^*\bigwedge^r E$ is ample.
\end{proof}

\begin{lemma}\label{thm1-l2}
Assume that $\theta_{E,r}$ defined in \eqref{e4} satisfies the inequality
$\theta_{E,r}\, <\, 0$. Then ${\mathcal O}_{{\rm
Gr}_r(E)}(1)$ is not nef.
\end{lemma}

\begin{proof}
Consider the strongly semistable vector bundle $V_t/V_{t-1}$
(see \eqref{e4}). Given any real number $\epsilon\, >\,0$,
and any $s\, \in\, [1\, ,\text{rank}(V_t/V_{t-1})]$,
there exists an irreducible smooth projective curve $Y$, a
nonconstant morphism
$$
f\, :\, Y\, \longrightarrow\, X\, ,
$$
and a subbundle
\begin{equation}\label{W}
W\, \subset\, f^*(V_t/V_{t-1})
\end{equation}
of rank $s$, such that
\begin{equation}\label{W1}
\mu(V_t/V_{t-1}) - \frac{\mu(W)}{\text{degree}(f)}\,=\,
\frac{\mu(f^*(V_t/V_{t-1})) -\mu(W)}{\text{degree}(f)}\,
<\, \epsilon
\end{equation}
(see \cite[p. 525, Theorem 4.1]{PS}). Set
$$
s\,=\, r- \text{rank}(((F^\delta_X)^*E)/V_t)
~\, ~\text{and~}\,~ \epsilon \,=\, -\frac{\theta_{E,r}}{2s}\, .
$$

Let $Q$ be the quotient of $f^*(F^\delta_X)^*E$ defined
by the composition
$$
f^*(F^\delta_X)^*E\, \longrightarrow\,
f^*(((F^\delta_X)^*E)/V_{t-1}) \, \longrightarrow\,
f^*(((F^\delta_X)^*E)/V_{t-1})/W\, ,
$$
where $f$ and $W$ are as in \eqref{W} for the above
choices of $s$ and $\epsilon$. Note that
$$
\text{degree}(Q)\,=\, \text{degree}(f)\cdot
\text{degree}(((F^\delta_X)^*E)/V_t)+
(\text{degree}(f)\cdot\text{degree}(V_t/V_{t-1})-
\text{degree}(W))\, .
$$
Hence from \eqref{e4},
$$
\text{degree}(Q)\,=\,\text{degree}(f)(\theta_{E,r}+(\mu(V_t/V_{t-1})
-\frac{\mu(W)}{\text{degree}(f)})\cdot s)\, .
$$
But from \eqref{W1}, we have $\mu(V_t/V_{t-1}) 
-\mu(W)/\text{degree}(f)\, <\,\epsilon$. Consequently,
\begin{equation}\label{Q}
\text{degree}(Q)\,< \, 0\, .
\end{equation} 

The quotient bundle $f^*(F^\delta_X)^*E\, \longrightarrow\, Q$ of 
rank $r$ defines a morphism
$$
\phi\, :\, Y\, \longrightarrow\, \text{Gr}_r((F^\delta_X)^*E)
\,=\, (F^\delta_X)^*\text{Gr}_r(E)\, ,
$$
where $\text{Gr}_r((F^\delta_X)^*E)$ is the Grassmann bundle
parametrizing all $r$ dimensional quotients of the fibers of
$(F^\delta_X)^*E$, and $(F^\delta_X)^*\text{Gr}_r(E)$ is the pullback
of the fiber bundle $\text{Gr}_r(E)\,\longrightarrow\, X$ using
the morphism $F^\delta_X$. Consider the commutative diagram
\begin{equation}\label{l1}
\begin{matrix}
(F^\delta_X)^*\text{Gr}_r(E) &\stackrel{\beta}{\longrightarrow}
& \text{Gr}_r(E)\\
\Big\downarrow && \Big\downarrow  \\
X &\stackrel{F^\delta_X}{\longrightarrow}& X
\end{matrix}
\end{equation}
of morphisms.
We have $\beta^*{\mathcal O}_{\text{Gr}_r(E)}(1)\,=\,
{\mathcal O}_{\text{Gr}_r((F^\delta_X)^*E)}(1)$, where
${\mathcal O}_{\text{Gr}_r((F^\delta_X)^*E)}(1)$ is the
tautological line bundle, and $\beta$ is the morphism in
\eqref{l1}. Hence from the definition
of $\phi$ it follows immediately that
$$
(\beta\circ\phi)^*{\mathcal O}_{\text{Gr}_r(E)}(1)\,=\,
\bigwedge\nolimits^r Q\, .
$$
Now from \eqref{Q} we conclude that ${\mathcal
O}_{\text{Gr}_r(E)}(1)$ is not nef.
\end{proof}

\begin{lemma}\label{thm1-l3}
Assume that $\theta_{E,r}\, =\, 0$ (defined in \eqref{e4}). Then 
${\mathcal O}_{{\rm Gr}_r(E)}(1)$ is nef but not ample.
\end{lemma}

\begin{proof}
The proof that ${\mathcal O}_{\text{Gr}_r(E)}(1)$ is nef
is very similar to the proof of Lemma \ref{thm1-l1}.

We know that $\bigwedge^r E$ is nef if and only if
$(F^\delta_X)^*\bigwedge^r E$ is nef
\cite[p. 360, Proposition 2.3]{Fu}
and \cite[p. 360, Proposition 2.2]{Fu}.
Consider the vector bundles $V_{\underline a}$ in \eqref{ua}. We
noted earlier that
$V_{\underline a}$ is strongly semistable. The condition
that $\theta_{E,r}\, =\, 0$ implies that
$$
\text{degree}(V_{\underline a})\, \geq\, 0\, .
$$

A strongly semistable vector bundle $W$ over $X$ of nonnegative
degree is nef. To prove this, take any morphism
$$
\psi\, :\, Y\, \, \longrightarrow\, {\mathbb P}(W)\, ,
$$
where $Y$ is an irreducible smooth projective curve. Let
$h\, :\, {\mathbb P}(W)\, \longrightarrow\, X$ be the natural
projection. The pullback $\psi^*h^* W$ is semistable because
$W$ is strongly semistable. Since $\psi^*{\mathcal O}_{{\mathbb 
P}(W)}(1)$ is a quotient of $\psi^*h^* W$, and $\text{degree}
(\psi^*h^* W)\, \geq\, 0$, we conclude that
$\text{degree}(\psi^*{\mathcal O}_{{\mathbb P}(W)}(1))\, \geq\, 0$.
Hence ${\mathcal O}_{{\mathbb P}(W)}(1)$ is nef, meaning
$W$ is nef.

The above observation implies that the
vector bundle $V_{\underline a}$ is nef.

Since each successive quotient of the filtration of $(F^\delta_X)^*\bigwedge^r E$
is nef (as they are of the form $V_{\underline a}$), from Lemma \ref{lem1} we know
that $(F^\delta_X)^*\bigwedge^r E$ is nef. We noted earlier that
$\bigwedge^r E$ is nef if $(F^\delta_X)^*\bigwedge^r E$ is so.
Now using \eqref{rho} and \eqref{rho1} we conclude that 
${\mathcal O}_{\text{Gr}_r(E)}(1)$ is nef.

To complete the proof of the lemma we need to show
that ${\mathcal O}_{\text{Gr}_r(E)}(1)$ is not 
ample.

Consider $V_t/V_{t-1}$ in \eqref{e4}. Let
\begin{equation}\label{f}
f\, :\, {\rm Gr}_s(V_t/V_{t-1})\, \longrightarrow\, X
\end{equation}
be the Grassmann bundle 
parametrizing quotients of the fibers of $V_t/V_{t-1}$ of
dimension
\begin{equation}\label{s}
s\,:=\, r- \text{rank}(((F^\delta_X)^*E)/V_t)\, .
\end{equation}
Let
\begin{equation}\label{gamma}
\gamma\, :\, {\rm Gr}_s(V_t/V_{t-1})\, \longrightarrow\,
\text{Gr}_r((F^\delta_X)^*E)
\end{equation}
be the morphism of fiber bundles over $X$ that sends any
quotient $q\, :\, (V_t/V_{t-1})_x\, \longrightarrow\, Q$
to the quotient defined by the composition
$$
((F^\delta_X)^*E)_x\, \longrightarrow\, (((F^\delta_X)^*E)/V_{t-1})_x \,\longrightarrow\,
((((F^\delta_X)^*E)/V_{t-1})_x)/\text{kernel}(q)\, .
$$
To define $\gamma$ using the universal property of a Grassmannian, let
$$
f^*(V_t/V_{t-1})\, \stackrel{\widetilde{q}}{\longrightarrow}\, {\mathcal Q}
\,\longrightarrow\, 0
$$
be the universal quotient bundle of rank $s$
over ${\rm Gr}_s(V_t/V_{t-1})$. Now consider the diagram of homomorphisms
$$
\begin{matrix}
&& \text{kernel}(\widetilde{q}) &\hookrightarrow & V_t/V_{t-1}
&\stackrel{\widetilde{q}}{\longrightarrow} & {\mathcal Q}\\
&& \bigcap && \bigcap \\
f^*(F^\delta_X)^*E&\stackrel{\widehat{q}}{\longrightarrow} &
((F^\delta_X)^*E)/V_{t-1} & = & ((F^\delta_X)^*E)/V_{t-1} \\
&& ~ \Big\downarrow h\\
&& (((F^\delta_X)^*E)/V_{t-1})/\text{kernel}(\widetilde{q})
\end{matrix}
$$
Note that $\text{rank}((((F^\delta_X)^*E)/V_{t-1})/\text{kernel}(\widetilde{q}))
\,=\, r$ by \eqref{s}. Let
$$
\widetilde{\gamma}\, :\, {\rm Gr}_s(V_t/V_{t-1})\,\longrightarrow\,
\text{Gr}_r(f^*(F^\delta_X)^*E)\,=\, {\rm Gr}_s(V_t/V_{t-1})\times_X
\text{Gr}_r((F^\delta_X)^*E)
$$
be the morphism representing the surjective homomorphism $h\circ \widehat{q}$
in the above diagram. The morphism $\gamma$ in \eqref{gamma} is the composition
of $\widetilde{\gamma}$ with the natural projection
${\rm Gr}_s(V_t/V_{t-1})\times_X\text{Gr}_r((F^\delta_X)^*E)
\,\longrightarrow\, \text{Gr}_r((F^\delta_X)^*E)$.

The morphism $\gamma$ in \eqref{gamma} is clearly an embedding. 
Define the line bundle
$$
{\mathcal L} \, :=\, \det(((F^\delta_X)^*E)/V_t)\,=\,
\bigotimes_{i=t+1}^d\bigwedge\nolimits^{{\rm 
rank}(V_i/V_{i-1})} (V_i/V_{i-1})
$$
on $X$. We note that
\begin{equation}\label{f2}
\gamma^*{\mathcal O}_{\text{Gr}_r((F^\delta_X)^*E)}(1)\,=\, {\mathcal
O}_{{\rm Gr}_s(V_t/V_{t-1})}(1)\otimes f^*{\mathcal L}\, ,
\end{equation}
where ${\mathcal O}_{{\rm Gr}_s(V_t/V_{t-1})}(1)\, \longrightarrow
\,{\rm Gr}_s(V_t/V_{t-1})$ is the tautological line bundle.

For any integer $n$, the
line bundles ${\mathcal O}_{\text{Gr}_r((F^\delta_X)^*E)}(1)^{\otimes n}$
and ${\mathcal O}_{{\rm Gr}_s(V_t/V_{t-1})}(1)^{\otimes n}$
will be denoted by ${\mathcal O}_{\text{Gr}_r((F^\delta_X)^*E)}(n)$
and ${\mathcal O}_{{\rm Gr}_s(V_t/V_{t-1})}(n)$ respectively.

Assume that ${\mathcal O}_{\text{Gr}_r(E)}(1)$ is ample. Since
$F^\delta_X$ is a finite morphism, this implies that
${\mathcal O}_{\text{Gr}_r((F^\delta_X)^*E)}(1)$ is ample.
Therefore, the pullback
$\gamma^*{\mathcal O}_{\text{Gr}_r((F^\delta_X)^*E)}(1)$ is
ample because
$\gamma$ is an embedding. Hence for sufficiently large positive
$n$, we have
\begin{equation}\label{p}
\dim H^0({\rm Gr}_s(V_t/V_{t-1}),\, \gamma^*{\mathcal 
O}_{\text{Gr}_r((F^\delta_X)^*E)}(n))
\, =\, cn^{d_0} +\sum_{j=0}^{d_0-1} a_j n^j
\end{equation}
with $c \, >\, 0$, where $d_0\, =\, \dim {\rm Gr}_s(V_t/
V_{t-1})$.

For convenience, the integer ${\rm rank}(V_t/V_{t-1})$ will be denoted by $r_t$.

Let $K^{-1}_f\, :=\, K^{-1}_{{\rm Gr}_s(V_t/V_{t-1})}\bigotimes f^*
K_X$ be the relative anti-canonical line bundle for the projection
$f$ in \eqref{f}. We have,
\begin{equation}\label{K}
K^{-1}_f\,=\, {\mathcal O}_{{\rm Gr}_s(V_t/V_{t-1})}(r_t)\otimes 
((\bigwedge\nolimits^{r_t} (V_t/V_{t-1}))^{\otimes s})^*\, ,
\end{equation}
where $s$ is defined in \eqref{s}. The given condition that $\theta_{E,r}
\,=\, 0$ implies that
$$
-s\cdot \text{degree}(V_t/V_{t-1})\,=\, r_t\cdot
\text{degree}(((F^\delta_X)^*E)/V_t)\, .
$$
Hence the two line bundles
$((\bigwedge\nolimits^{r_t} (V_t/V_{t-1}))^{\otimes s})^*$
and ${\mathcal L}^{\otimes r_t}$ differ by tensoring with
a line bundle of degree zero. Therefore, from \eqref{K}
we conclude that
$$
({\mathcal O}_{{\rm Gr}_s(V_t/V_{t-1})}(1)\otimes {\mathcal L})^{
\otimes r_t}\,=\, K^{-1}_f\otimes f^*{\mathcal L}_0\, ,
$$
where ${\mathcal L}_0$ is a line bundle on $X$ of degree zero.
Now, from \eqref{f2},
\begin{equation}\label{p1}
\gamma^*{\mathcal O}_{\text{Gr}_r((F^\delta_X)^*E)}(r_t)\,=\, 
K^{-1}_f\otimes f^*{\mathcal L}_0\, .
\end{equation}

{}From the projection formula, and \eqref{p1},
\begin{equation}\label{p2}
H^0({\rm Gr}_s(V_t/V_{t-1}),\, \gamma^*{\mathcal O}_{\text{Gr}_r((F^\delta_X)^*E)}
(n\cdot r_t))\,=\, H^0(X,\, (f_*(K^{-1}_f)^{\otimes n})
\otimes {\mathcal L}^{\otimes n}_0)\, .
\end{equation}

We will show that the line bundle $\det (f_*(K^{-1}_f)^{\otimes
n})\,\longrightarrow\, X$ is trivial. For that, let $F_{\text{GL}_{r_t}}$ be
the principal $\text{GL}_{r_t}(k)$--bundle on $X$ defined by
the vector bundle $V_t/V_{t-1}$; the fiber of $F_{\text{GL}_{r_t}}$
over any point $x\, \in\, X$ is the space of all linear
isomorphisms from $k^{\oplus r_t}$ to the fiber $(V_t/V_{t-1})_x$.
Let $F_{\text{PGL}_{r_t}}\, :=\, F_{\text{GL}_{r_t}}/{\mathbb G}_m$
be the corresponding principal $\text{PGL}_{r_t}(k)$--bundle.
The vector bundle $f_*(K^{-1}_f)^{\otimes n}$ is the one associated
to the principal $\text{PGL}_{r_t}(k)$--bundle
$F_{\text{PGL}_{r_t}}$ for the $\text{PGL}_{r_t}(k)$--module
$H^0(\text{Gr}_s(k^{\oplus r_t}),\, 
(K^{-1}_{\text{Gr}_s(k^{\oplus r_t})})^{\otimes n})$ (the action 
of $\text{PGL}_{r_t}(k)$ on the space of sections is given by the 
standard action of $\text{PGL}_{r_t}(k)$ on
$\text{Gr}_s(k^{\oplus r_t})$. Since
$\text{PGL}_{r_t}(k)$ does not have any nontrivial character, the
line bundle $\det (f_*(K^{-1}_f)^{\otimes n})$ associated to
$F_{\text{PGL}_{r_t}}$ for the $\text{PGL}_{r_t}(k)$--module
$\bigwedge^{\rm top}H^0(\text{Gr}_s(k^{\oplus r_t}),\,
(K^{-1}_{\text{Gr}_s(k^{\oplus r_t})})^{\otimes n})$ is trivial.

As $\det (f_*(K^{-1}_f)^{\otimes n})$ is trivial and
$\text{degree}({\mathcal L})\,=\,0$,
$$
\text{degree}(((f_*(K^{-1}_f)^{\otimes n})
\otimes {\mathcal L}^{\otimes n}_0)\,=\, 0\, .
$$
Since $V_t/V_{t-1}$ is strongly semistable, the corresponding
principal $\text{GL}_{r_t}(k)$--bundle $F_{\text{GL}_{r_t}}$
is strongly semistable. Therefore, the associated vector
bundle $f_* (K^{-1}_f)^{\otimes n}$ is also semistable
(see \cite[p. 285, Theorem 3.18]{RR} and
\cite[p. 288, Theorem 3.23]{RR}). This implies that
$(f_*(K^{-1}_f)^{\otimes n})\bigotimes {\mathcal L}^{\otimes n}_0$ is semistable.

For a semistable vector bundle $\mathcal V$ on $X$ of degree zero, any nonzero
section $\sigma\,:\, {\mathcal O}_X\,\longrightarrow\, \mathcal V$ is nowhere
vanishing. Indeed, this follows immediately from the semistability condition that
the line bundle of $\mathcal V$ generated by the image of $\sigma$ is of
nonpositive degree. Consequently,
$$
\dim H^0(X,\, {\mathcal V})\, \leq\, {\rm rank}({\mathcal V})\, .
$$

Since $(f_*(K^{-1}_f)^{\otimes n})\bigotimes {\mathcal L}^{\otimes n}_0$ is semistable
of degree zero, we have
\begin{equation}\label{eqinq}
\dim H^0(X,\, (f_*(K^{-1}_f)^{\otimes n})\otimes {\mathcal L}^{\otimes n}_0)
\, \leq\,{\rm rank}((f_*(K^{-1}_f)^{\otimes n})
\otimes {\mathcal L}^{\otimes n}_0)\,=\,
{\rm rank}((f_*(K^{-1}_f)^{\otimes n}))
\end{equation}
for all $n\, >\, 0$.

We have $R^jf_*((K^{-1}_f)^{\otimes n})\,=\, 0$ for
$j\, , n\, \geq\, 1$. Hence from the Riemann--Roch theorem for the restriction
$(K^{-1}_f)^{\otimes n}\vert_{f^{-1}(x)}$,
$x\,\in\, X$, we conclude that ${\rm rank}((f_*(K^{-1}_f)^{\otimes n}))$
is a polynomial of degree at most $d_0-1$ (which is the dimension of
the fibers of $f$). Therefore, using \eqref{p2} and \eqref{eqinq}
we conclude that
$$
\dim H^0({\rm Gr}_s(V_t/V_{t-1}),\, \gamma^*{\mathcal 
O}_{\text{Gr}_r((F^\delta_X)^*E)} (n\cdot {\rm rank}( V_t/V_{t-1})))
$$
is a polynomial of degree at most $d_0-1$. But this contradicts
\eqref{p}.

We assumed that ${\mathcal O}_{\text{Gr}_r(E)}(1)$ is ample, and were
led to the above contradiction. Therefore, we conclude that ${\mathcal
O}_{\text{Gr}_r(E)}(1)$ is not ample. This completes the proof
of the lemma.
\end{proof}

Lemma \ref{thm1-l1}, Lemma \ref{thm1-l2} and Lemma \ref{thm1-l3} together
give he following:

\begin{theorem}\label{thm1}
If $\theta_{E,r}\, >\, 0$, then the line bundle
${\mathcal O}_{{\rm Gr}_r(E)}(1)\, \longrightarrow\,
{\rm Gr}_r(E)$ in \eqref{f1} is ample.

If $\theta_{E,r}\, =\, 0$, then
${\mathcal O}_{{\rm Gr}_r(E)}(1)$ is nef but not ample.

If $\theta_{E,r}\, <\, 0$, then ${\mathcal O}_{{\rm
Gr}_r(E)}(1)$ is not nef.
\end{theorem}

\section{The nef cone of $\text{G}\mathtt{r}_r(E)$}\label{s-n-c}

In this section we will compute the nef cone of $\text{Gr}_r(E)$
using Theorem \ref{thm1}. Being a closed cone, it is generated by its
boundary. For notational reasons, it will be
convenient to treat the cases of characteristic zero and
positive characteristic separately.

For a smooth projective variety $Z$, the real N\'eron--Severi group
$\text{NS}(Z)_{\mathbb R}$ is defined to be
\begin{equation}\label{N-Sg}
\text{NS}(Z)_{\mathbb R}\,:=\, ({\rm Pic}(Z)/{\rm Pic}^0(Z))\otimes_{\mathbb Z}
\mathbb R\, ,
\end{equation}
where ${\rm Pic}^0(Z)$ is the connected component, containing the identity
element, of the Picard group ${\rm Pic}(Z)$ of $Z$.
 
\subsection{Characteristic is zero} In this case, the number
$\delta$ in \eqref{e4} is zero.

As in \eqref{vp}, $\varphi$ is the projection of
$\text{Gr}_r(E)$ to $X$. Fix a line bundle $L_1$ over $X$
of degree one. The line bundle $\varphi^*L_1$ will be
denoted by ${\mathcal L}$. The real N\'eron--Severi group 
$\text{NS}(\text{Gr}_r(E))_{\mathbb R}$
is freely generated by $\mathcal L$ and 
${\mathcal O}_{\text{Gr}_r(E)}(1)$.

Although $\theta_{E,r}$ in \eqref{e4} need not be an integer, we note that
${\mathcal L}^{\otimes -\theta_{E,r}}$ is well defined as an
element of $\text{NS}(\text{Gr}_r(E))_{\mathbb R}$ because
$\theta_{E,r}\, \in\, \mathbb Q$.

\begin{proposition}\label{prop-e1}
The boundary of the nef cone in ${\rm NS}({\rm Gr}_r(E))_{\mathbb R}$
is given by $\mathcal L$ and ${\mathcal O}_{{\rm Gr}_r(E)}
(1)\bigotimes {\mathcal L}^{\otimes -\theta_{E,r}}$.
\end{proposition}

\begin{proof}
We will first show that it is enough to treat the case where $\theta_{E,r}$
is a multiple of $r$. In fact, this argument is standard (see
\cite[p. 23, Lemma 6.2.8]{Laz}). However, we describe the details for
completeness.

Write
$$
\theta_{E,r}\,=\, \frac{p_1r}{q_1}\, ,
$$
where $p_1$ and $q_1$ are integers with $q_1\, >\, 0$. Take a pair $(Y\, ,f)$, where
$Y$ is an irreducible smooth projective curve, and $f$ is a morphism from
$Y$ to $X$, such that $\text{degree}(f)$ is a multiple of $q_1$.
The natural map
$$
\gamma\, :\, \text{Gr}_r(f^*E)\, \longrightarrow\, \text{Gr}_r(E)
$$
produces an isomorphism between $\text{NS}(\text{Gr}_r(E))_{\mathbb R}$
and $\text{NS}(\text{Gr}_r(f^*E))_{\mathbb R}$. This isomorphism
preserves the nef cones. Therefore, it is enough to prove the
proposition for $(Y\, , f^*E)$.
Note that $\theta_{f^*E,r}\,=\, \frac{{\rm degree}(f)p_1r}{q_1}$
is a multiple of $r$.

Hence we can assume that $\theta_{E,r}/r$ is an integer.

Consider the vector bundle
$$
F\, :=\, E\otimes L^{\otimes -\frac{\theta_{E,r}}{r}}_1\, .
$$
Note that $\text{Gr}_r(E)\,=\, \text{Gr}_r(F)$. From 
\eqref{e4} and the definition of $F$ it follows immediately
that
$$
\theta_{F,r}\,=\, 0\, .
$$

Since $\theta_{F,r}\,=\, 0$, from the second part of
Theorem \ref{thm1} we know that the nef cone in $\text{NS}
(\text{Gr}_r(F))_{\mathbb R}$ is generated by ${\mathcal 
O}_{\text{Gr}_r(F)}(1)$ and $\mathcal L$ (it is
considered as a line bundle on $\text{Gr}_r(F)$
using the identification of
$\text{Gr}_r(F)$ with $\text{Gr}_r(E)$). The proposition
follows immediately from this description of the nef cone
in $\text{NS} (\text{Gr}_r(F))_{\mathbb R}$ using the
identification of $\text{Gr}_r(F)$ with $\text{Gr}_r(E)$.
\end{proof}

\begin{remark}\label{rem1}
{\rm We note that the two generators of the nef cone given in
Proposition \ref{prop-e1} lie in the rational N\'eron--Severi group
${\rm NS}({\rm Gr}_r(E))_{\mathbb Q}\,:=\,({\rm Pic}(Z)/{\rm Pic}^0(Z))
\otimes_{\mathbb Z} \mathbb Q$.}
\end{remark}

\subsection{Characteristic is positive}

Let $p\, >\, 0$ be the characteristic of $k$.
Consider $\delta$ in \eqref{e4}.
Let $\varphi_1\, :\, \text{Gr}_r((F^\delta_X)^*E)\,
\longrightarrow\, X$ be the natural projection. Define the
line bundle
$$
{\mathcal L}_1\, :=\, \varphi^*_1 L_1\, \longrightarrow\, X\, ,
$$
where $L_1$ is a fixed line bundle on $X$ of degree one.

\begin{lemma}\label{lem-1}
The nef cone in ${\rm NS}({\rm Gr}_r((F^\delta_X)^*E))_{
\mathbb R}$ (defined in \eqref{N-Sg}) is generated by ${\mathcal L}_1$ and ${\mathcal 
O}_{{\rm Gr}_r((F^\delta_X)^*E)}(1)\bigotimes {\mathcal L}^{\otimes 
-\theta_{(F^\delta_X)^*E,r}}_1$.
\end{lemma}

\begin{proof}
The proof is exactly identical to the proof of Proposition
\ref{prop-e1}. We refrain from repeating it.
\end{proof}

As in \eqref{vp}, the projection of $\text{Gr}_r(E)$ to $X$
will be denoted by $\varphi$. Define
$$
{\mathcal L}\, :=\, \varphi^*L_1\, .
$$

\begin{proposition}\label{prop-e2}
The boundary of the nef cone in ${\rm NS}({\rm Gr}_r(E))_{\mathbb R}$
is given by $\mathcal L$ and ${\mathcal O}_{{\rm Gr}_r(E)}
(p^\delta)\bigotimes {\mathcal L}^{\otimes - \theta_{(F^\delta_X)^*E,r}}$.
\end{proposition}

\begin{proof}
Consider the commutative diagram of morphisms in \eqref{l1}.
The morphism $\beta$ in this diagram produces an isomorphism
between $\text{NS}(\text{Gr}_r(E))_{\mathbb R}$ and
$\text{NS}(\text{Gr}_r((F^\delta_X)^*E))_{\mathbb R}$.
This isomorphism preserves the nef cones.

We have $\beta^*{\mathcal O}_{{\rm Gr}_r(E)} (1)\,=\,
{\mathcal O}_{{\rm Gr}_r((F^\delta_X)^*E)} (1)$, and
$(F^\delta_X)^*L_1\, =\, L^{\otimes p^\delta}_1$.
Hence the proposition follows from Lemma \ref{lem-1}.
\end{proof}

\begin{remark}\label{rem2}
{\rm The two generators of the nef cone
given in Proposition \ref{prop-e2} lie
in ${\rm NS}({\rm Gr}_r(E))_{\mathbb Q}$.}
\end{remark}

\section{The nef cone of flag bundles}\label{se-n-g}

Fix integers
$$
0\, <\, r_1 \, <\, r_2\, <\, \cdots \, <\, r_{\nu-1}
\, <\, r_\nu \, <\, \text{rank}(E)\, .
$$
Let
$$
\Phi\, :\, {\rm Fl}(E) \, \longrightarrow\, X
$$
be the corresponding flag bundle; so for any $x\, \in\, X$,
the fiber $\Phi^{-1}(x)$ parametrizes all filtrations of
linear subspaces
\begin{equation}\label{fl2}
E_x \, \supset\, S_1 \, \supset \, S_2\, \supset\, \cdots \,
\supset\, S_{\nu-1} \, \supset \, S_\nu 
\end{equation}
such that $\dim E_x - \dim S_i \,=\, r_i$ for all
$i\, \in\, [1\, , \nu]$.

For each $i\, \in\, [1\, , \nu]$, let $\text{Gr}_{r_i}(E)$ be
the Grassmann bundle over $X$ parametrizing all the $r_i$ 
dimensional quotients of the fibers of $E$. Let
\begin{equation}\label{fl3}
\phi_i\, :\, {\rm Fl}(E) \, \longrightarrow\, \text{Gr}_{r_i}(E)
\end{equation}
be the natural projection that sends any filtration as in
\eqref{fl2} to $E_x/S_i$. Let
$$
\omega_i\, \in\, {\rm NS}({\rm Gr}_{r_i}(E))_{\mathbb R}
$$
be the element ${\mathcal O}_{{\rm Gr}_{r_i}(E)}
(1)\bigotimes {\mathcal L}^{\otimes -\theta_{E,r_i}}$
(respectively, ${\mathcal O}_{{\rm Gr}_{r_i}(E)}
(p^\delta)\bigotimes {\mathcal L}^{\otimes -
\theta_{(F^\delta_X)^*E,r_i}}$) in Proposition
\ref{prop-e1} (respectively, Proposition \ref{prop-e2})
if the characteristic of $k$ is 
zero (respectively, positive). Define
$$
\widetilde{\omega}_i\, :=\, \phi^*_i\omega_i\, \in
{\rm NS}({\rm Fl}(E))_{\mathbb R}\, ,
$$
where $\phi_i$ is the projection in \eqref{fl3}.

\begin{theorem}\label{th-fl}
The nef cone in ${\rm NS}({\rm Fl}(E))_{\mathbb R}$ is generated
by $\{\widetilde{\omega}_i\}_{i=1}^\nu\bigcup \Phi^* {\mathcal L}'$,
where ${\mathcal L}'$ is a line bundle over $X$ of degree one.
\end{theorem}

\begin{proof}
The dimension of the $\mathbb R$--vector space
${\rm NS}({\rm Fl}(E))_{\mathbb R}$ is $\nu+1$,
and the vector space is generated by $\{\widetilde{\omega}_i\}_{i=1}^\nu
\bigcup\Phi^* {\mathcal L}'$. We note that ${\mathcal L}'$ and all
$\widetilde{\omega}_i$ are nef.

Fix any point $x\, \in\, X$. For each $i\, \in\, [1\, , \nu]$,
define
$$
\widetilde{\omega}_{x,i}\, :=\, 
\widetilde{\omega}_i\vert_{\Phi^{-1}(x)}\, \in\, 
{\rm NS}(\Phi^{-1}(x))_{\mathbb R}\, .
$$
The dimension of the $\mathbb R$--vector space
${\rm NS}(\Phi^{-1}(x))_{\mathbb R}$ is $\nu$. It
is known that the nef cone of ${\rm NS}(\Phi^{-1}(x))_{\mathbb 
R}$ is generated by $\{\widetilde{\omega}_{x,i}\}_{i=1}^\nu$
(see \cite[p. 187, Theorem 1]{Br} for a general result).
In view of this, the theorem follows from
Proposition \ref{prop-e1} (respectively, Proposition 
\ref{prop-e2}) when the characteristic of $k$ is zero
(respectively, positive).
\end{proof}

\begin{remark}\label{rem3}
{\rm All the elements of the generating set of the nef cone in
${\rm NS}({\rm Fl}(E))_{\mathbb R}$ given in Theorem \ref{th-fl} lie
in ${\rm NS}({\rm Fl}(E))_{\mathbb Q}$.}
\end{remark}


\end{document}